\documentclass[twoside]{amsart}
\usepackage{graphicx}
\usepackage{amssymb,amsfonts,amstext,amsmath,amsthm}
\newcommand{\ben}{\begin{enumerate}}
\newcommand{\een}{\end{enumerate}}
\newcommand{\disp}{\displaystyle}
\input{amssym.def}
\input{amssym.tex}
\newtheorem{thm}{Theorem}[section]

\newtheorem{cor}[thm]{Corollary}
\theoremstyle{definition}
\newtheorem{definition}[thm]{Definition}
\begin{document}
\title[TAT labeling of Ladders, Prisms and Generalised graphs]{Totally Antimagic Total labeling of Ladders, Prisms and Generalised Pertersen graphs}
\author{D. O. A. Ajayi, A. D. Akwu}
%\label{pageinit}
\date{}

\address{Deborah Olayide A. AJAYI \hfill\break\indent Department
of Mathematics, University of Ibadan, Ibadan, Nigeria}
\address{Abolape D. Akwu \hfill\break\indent Department of Mathematics, Federal University of Agriculture, Makurdi, Nigeria}
\email{adelaideajayi@yahoo.com, abolaopeyemi@yahoo.co.uk }
%\email{ abolaopeyemi@yahoo.co.uk}

\begin{abstract}
Given a graph $G$, a total labeling on $G$ is called \emph{edge-antimagic total (respectively, vertex-antimagic total}) if all edge-weights (respectively, vertex-weights) are pairwise distinct. If a labeling on $G$ is simultaneously edge-antimagic total and vertex-antimagic total, it is called a \emph{totally antimagic total labeling}. A graph that admits totally antimagic total labeling is called a \emph{totally antimagic total graph}.
In this paper, we prove that ladders, prisms and generalised Pertersen graphs are totally antimagic total graphs. We also show that the chain graph of totally antimagic total graphs is a totally antimagic total graph.
\end{abstract}
\maketitle

\section{Introduction}
We consider finite, undirected graphs without loops and multiple edges. For a graph $G$, $V(G)$ and $E(G)$ denotes the vertex-set and the edge-set respectively. A $(p,q)$ graph $G$ is a graph such that $|V(G)|=p$ and $|E(G)|=q$. We refer the reader to \cite{WA} and \cite{WE} for all other terms and notation not provided in this paper.

A labeling of a graph $G$ is any mapping that sends some set of graph elements to a set of non-negative integers. If the domain is the vertex-set or the edge-set, the labelings are called \emph{vertex labeling or edge labeling} respectively. Moreover, if the domain is $V(G)\cup E(G)$ then the labeling is called \emph{total labeling}.

Let $f$ be a vertex labeling of a graph $G$, we define the edge-weight of $uv\in E(G)$ to be $wt_f(uv)=f(u)+f(v)$. If $f$ is a total labeling, then the edge-weight of $uv$ is $wt_f(uv)=f(u)+f(v)+f(uv)$. The vertex-weight of a vertex $v$, $v\in E(G)$ is defined by $$wt_f(v)=\sum_{u\in N(v)}f(uv)+f(v)$$
where $N(v)$ is the set of the neighbors of $V$. If the vertices are labeled with the smallest posssible numbers i.e. $f(V(G))=\{1,2,3,...,p\}$, then the total labeling is called \emph{super}.

A labeling $f$ is called \emph{edge-antimagic total(vertex-antimagic total)}, for short EAT(VAT), if all edge-weights (vertex-weights) are pairwise distinct. A graph that admits EAT (VAT) labeling is called an EAT (VAT) graph. If the edge-weights (vertex-weights) are all the same, then the total labeling is called \emph{edge-magic total (vertex-magic total)}. For an edge labeling, a \emph{vertex-antimagic edge} (VAE) labeling is a labeling whereby a vertex-weight is the sum of the labels of all edges incident with the vertex.

In 1990, Harsfield and Ringel \cite{HR} introduced the concept of an antimagic labeling of graphs whereby they conjectured that every tree except $P_2$ has a VAE labeling. This conjecture was proved to be true for all graphs having minimum degree $log|V(G)|$ by Alon \emph{et al} \cite{AK}.
If a VAE labeling satisfies the condition that the set of all the vertex-weights is $\{a,a+d,...,a+(p-1)d\}$ where $a>0$ and $d\geq 1$ are two fixed integers, then the labeling is called an \emph{(a,d)-VAE labeling}. For further results on graph labeling see \cite{BAM}, \cite{GA} and \cite{MD}.
In \cite{MPJ}, Miller \emph{et al} proved that all graphs are (super) EAT. They also proved that all graphs are (super )VAT.
If the labeling is simultaneously vertex-antimagic total and edge-antimagic total, then it is referred to as \emph{totally-antimagic total} (TAT) labeling and a graph that admits such labeling is a \emph{totally-antimagic total} (TAT) graph. The definition of totally antimagic total labeling is a natural extension of the concept of totally magic labeling. In \cite{BE1}, Baca \emph{et al} deals with totally antimagic total graphs. They found totally-antimagic total labeling of some classes of graphs and proved that paths, cycles, stars, double-stars and wheels are totally antimagic total.  Moreover, they showed that a union of regular totally antimagic total graphs is a totally antimagic total graph. Also, in \cite{AJ}, Akwu and Ajayi showed that complete bipartite graphs are totally antimagic total graphs.

In this paper, we deal with totally antimagic total labeling of ladders, prisms and generalised Petersen graphs. We also show that the chain graphs obtained by concatenation of totally antimagic total graphs are totally antimagic total graphs.

First we provide some definitions which are related to the present work.
%\subsection{Definition}
\begin{definition}
\emph{Ladder} is a graph obtained by the cartesian product of path $P_n$ and path $P_2$ denoted by $L_n$, i.e. $L_n \simeq P_n\times P_2$ where $V(L_n)=\{u_iv_i:1\leq i\leq n\}$ and $E(L_n)=\{u_iu_{i+1},v_iv_{i+1}:1\leq i\leq n-1\}\cup \{u_iv_i:1\leq i\leq n\}$.
 \end{definition}
% Let $v_f(i)$, $i= 0,1$ denote the number of vertices of $G$ having label $i$  under $f$ and let $e_f(i)$ $i=0,1$ be the number of edges of $G$ having label $i$ under $f^*$.
%\end{definition}
\begin{definition}
A labeling $g$ is ordered (sharp ordered) if $wt_g(u)\leq wt_g(v)$ $(wt_g(u)<wt_g(v))$ holds for every pair of vertices $u,v\in G$ such that $g(u)<g(v)$. A graph that admits a sharp ordered labeling is called a \emph{(sharp) ordered } graph. Also, if the vertex set can be partitioned into $n$ sets such that each set is sharp ordered, then the graph is  referred to as \emph{weak ordered} graph.
\end{definition}
\begin{definition}
The \emph{prism} graph can be defined as the cartesian product $C_n\times P_2$ of a cycle on $n$ vertices with a path of length $2$. The vertex set is $V(C_n\times P_2)=\{u_iv_i:1\leq i\leq n\}$ and the edge set is
$E(C_n\times P_2)=\{u_iu_{i+1},v_iv_{i+1}:1\leq i\leq n\}\cup \{u_iv_i:1\leq i\leq n\}$
%$\{u_i u_{i+1},v_iv_{i+1}:1\leq i\leq n\} $
%$\cup \{u_iv_i:1\eq i\leq n\}$
where $i$ is calculated modulo $n$. The orders of the vertex set and the edge set are $2n$ and $3n$ respectively.
\end{definition}
\begin{definition}
The \emph{generalised Pertersen} graph  $P(n,m), \ n\geq 3$ and $1\leq m\leq \lfloor \frac{n-1}{2} \rfloor$ consists of an outer $n$-cycle $u_i,u_2,...,u_n$, a set of $n$ spokes $u_iv_i, 1\leq i\leq n$ and $n$ edges $v_iv_{i+m}$, $1\leq i\leq n$ with indices taken modulo $n$.
\end{definition}

%Since then several variants of cordial labeling have been introduced and investigated. % labeling like prime cordial labeling, A-cordial labeling,
%$H$-cordial labeling and product cordial labeling were also
%introduced as variants of cordial labeling.
%After this,
%A lot of graphs have also been studied for their cordiality. Shee and Ho \cite{SH} in 1993 investigated the
%cordiality of one-point union of $n$-copies of a graph.
%\subsection{Definition}
%For any integer $p>1$. A mapping $f:V(G)\rightarrow\{0,1\}$ is called a $\it{smarandachely\ p- product\ cordial}$ labeling if $|v_f(i)-v_f(j)|\leq 1$ and $|e_f(i)-e_f(j)|\leq 1$ for any $i,j\in \{0,1,...,p-1\}$, where $v_f(i)$ denotes the number of vertices labeled with $i$, $e_f(i)$ denotes the number of edges $xy$ with $f(x)f(y)\equiv i(mod\ p)$. Particularly, if $p=2$, i.e., a binary vertex labeling of graph $G$ with an induced edge labeling $f^*:E(G)\rightarrow \{0,1\}$ defined by $f^*(e=uv)=f(u)f(v)$, such a smarandachely $2$-product cordial labeling is called $\it{product\  cordial}$ labeling. A graph with product cordial labeling is called a product cordial graph.\\
%\begin{definition} A binary vertex labeling $f: V(G) \rightarrow \{0,1\}$
%of a graph $G$ that induces an edge labeling $f* E(G) \rightarrow \{0,1\}$ such that $f*(uv)=f(u)f(v)$ is called a product cordial labeling if .............
%\end{definition}

\section{Totally antimagic total graphs}
In this section, we prove that ladders, prism graphs and generalised Petersen graphs are totally antimagic total graphs.
\begin{thm}
The ladder graph $L_n$, $n\geq 2$ is a weak ordered super TAT.
\end{thm}
\begin{proof}
We denote the vertices of $L_n$ by the following symbols $V(L_n)=\{u_iv_i:1\leq i\leq n\}$ such that $E(L_n)=\{u_iu_{i+1},v_iv_{i+1}:1\leq i\leq n-1\}\cup \{u_iv_i:1\leq i\leq n\}$. Consider the labeling $g$ of $L_n$ as follows:
$$g(u_i)=\left\{\begin{array}{ll}
2i-1, & 1\leq i\leq \lceil\frac{n}{2}\rceil\\ \ \\
2(n-i+1), & \lceil\frac{n}{2} \rceil+1\leq i\leq n
\end{array}\right.$$
and
$$g(v_i)=\left\{\begin{array}{ll}
n+2i-1, & 1\leq i\leq \lceil\frac{n}{2}\rceil\\ \ \\
3n+2(1-i), & \lceil\frac{n}{2} \rceil+1\leq i\leq n
\end{array}\right.$$
Also, the labeling of the edges is as follows:
$$g(u_iv_i)=\left\{\begin{array}{ll}
2(n+i)-1, & 1\leq i\leq \lceil\frac{n}{2} \rceil\\ \ \\
2(2n-i+1), & \lceil\frac{n}{2} \rceil+1\leq i\leq n-1
\end{array}\right.$$
$$g(u_iu_{i+1})=\left\{\begin{array}{ll}
3n+2i-1, & 1\leq i\leq \lfloor\frac{n}{2} \rfloor\\ \ \\
5n-2i, & \lfloor\frac{n}{2} \rfloor+1\leq i\leq n-1
\end{array}\right.$$
$$g(v_iv_{i+1})=\left\{\begin{array}{ll}
2(2n+i-1), & 1\leq i\leq \lfloor\frac{n}{2} \rfloor\\ \ \\
2(3n-i)-1, & \lfloor\frac{n}{2} \rfloor+1\leq i\leq n-1
\end{array}\right.$$
The above labeling is super. Next, we show that the vertex-weights are pairwise distinct.\\
The vertex-weights of $V(L_n)$ are
$$wt_g(u_i)=g(u_i)+\sum_{u\in N(u)}g(u_iu)$$ where vertex $u$ is any vertex adjacent to vertex $u_i$.
For $i=1$ and $i=n$, we have the weights $$wt_g(u_1)= 3+5n$$ and $$wt_g(u_n)=6+5n.$$
Also, for $i=2,...,n-1$, we have the weight $$wt_g(u_i)=\left\{\begin{array}{lll} \ \\
2(4i+4n-3), & 2\leq i\leq \lfloor\frac{n}{2} \rfloor\\ \ \\
12n-3, & i=\lfloor\frac{n}{2} \rfloor+1\\ \ \\
2(4(2n-i)+3), & \lfloor \frac{n}{2} \rfloor +2 \leq i\leq n-1
\end{array}\right.$$
Furthermore,
$$wt_g(v_i)=\left\{\begin{array}{ll}
W_1, \\ \ \\
W_2
\end{array}\right.$$
where $$W_1=2(2i-1)+k(2n+i-1)-(t+3n)$$ with the following conditions:

for $i=1, k=1$ and $t=0$,

for $2\leq i\leq \lfloor \frac{n}{2} \rfloor, k=t=2$.\\
While $$W_2=2(2(1-i)+k(3n-i))+(t-k)+7n$$ with the following conditions:

$k=t=2$ whenever $\lfloor \frac{n}{2}\rfloor +2\leq i\leq n-1$,

$k=1$ and $t=2$ for $i=n$,

for $n$ even and $i= \frac{n}{2} +1$, $k=2$ and $t=1$,\\

for $n$ odd and  $i=\lfloor \frac{n}{2}\rfloor +1$, $k=2$ and $t=-3$.

In view of the above labeling, the weights of all the vertices are different, that is the labeling is vertex-antimagic total.

Now, we show that the edge-weights are pairwise distinct. The edge-weight of the edges under labeling $g$ is as follows:
$$wt_g(u_iu_{i+1})=\left\{\begin{array}{ll}
W_3, \\ \ \\
W_4
\end{array}\right.$$
where $$W_3=3(2i+n-1)+t$$ with the following conditions:

$t=2$ for $1\leq i\leq \lceil \frac{n}{2}\rceil -1$,

$t=1$ for $n$ even and $i=\frac{n}{2}$,

$t=-2$ for $n$ odd and $i=\lceil  \frac{n}{2}\rceil$,

while $$W_4=3(3n-2i)+2$$ for $\lceil \frac{n}{2}\rceil +1\leq i\leq n-1$.
Also, $$wt_g(v_iv_{i+1})=\left\{\begin{array}{ll}
W_5, \\ \ \\
W_6
\end{array}\right.$$
where $$W_5=2(3(n+i)-2)+t$$ with the following conditions:

$t=2$ whenever $1\leq i\leq \lceil \frac{n}{2}\rceil -1$,

$t=1$ whenever $n$ is even and $i=\frac{n}{2}$,

$t=-2$ 	whenever $i=\lceil \frac{n}{2} \rceil $ and $n$ is odd.
Also,
$$W_6=3(2(2n-i)+1)-2, \lceil \frac{n}{2} \rceil +1\leq i\leq n-1$$
In view of the above labeling, the edge-weights are pair-wise distinct, which implies that the labeling is edge-antimagic labeling, i.e. ladders are edge-antimagic total graphs. \ \\ \ \\
If we partition the vertex set into two, i.e. $u_i$ and $v_i$, $1\leq i\leq n$, each vertex set is sharp ordered which implies that the graph is a weak ordered graph. Therefore ladders are super weak ordered TAT graphs since they are both edge-antimagic total graphs and vertex-antimagic total graphs.
\end{proof}
The following theorem shows that prism graphs are TAT graphs.
\begin{thm}
The prism graph $C_n\times P_2$ is a super TAT graph for every $n>2$.
\end{thm}
\begin{proof}
Denote the vertices of prism graph $C_n\times P_2$ by $\{u_iv_i:1\leq i\leq n\}$ and the edges by $\{u_iu_{i+1},v_iv_{i+1}:1\leq i\leq n\}\cup \{u_iv_i:1\leq i\leq n\}$. Let $g$ be the labeling on $C_n\times P_2$. Define the labeling $g$ on the vertices as follows:
Whenever $i=1$, $g(u_1)=1$ and $g(v_1)=n+1$.
Also,
$$g(u_i)=\left\{\begin{array}{ll}
2(i-1), & 2\leq i\leq \lfloor\frac{n}{2} \rfloor +1\\ \ \\
2(n-i)+3, & \lfloor\frac{n}{2} \rfloor+2\leq i\leq n
\end{array}\right.$$
$$g(v_i)=\left\{\begin{array}{ll}
n+2(i-1), & 2\leq i\leq \lfloor\frac{n}{2} \rfloor +1\\ \ \\
3(n+1)-2i, & \lfloor\frac{n}{2} \rfloor+2\leq i\leq n
\end{array}\right.$$
From the labeling above, we have the labeling $g$ to be super.\\
Moreover, define the labeling $g$ on the edge set as follows:
$$g(u_iu_{i+1})=\left\{\begin{array}{ll}
2(n+i)-1, & 1\leq i\leq \lceil\frac{n}{2} \rceil\\ \ \\
2(2n+1-i, & \lceil\frac{n}{2} \rceil+1\leq i\leq n
\end{array}\right.$$
$$g(v_iv_{i+1})=\left\{\begin{array}{ll}
2(2n+i)-1, & 1\leq i\leq \lceil\frac{n}{2} \rceil\\ \ \\
2(3n-i+1), & \lceil\frac{n}{2} \rceil+1\leq i\leq n
\end{array}\right.$$
$$g(u_iv_i)=\left\{\begin{array}{ll}
2(2n-1)+3, & 2\leq i\leq \lfloor\frac{n}{2} \rfloor +1\\ \ \\
2(n-1+i), & \lfloor\frac{n}{2} \rfloor+2\leq i\leq n
\end{array}\right.$$
For $i=1$, we have $g(u_1v_1)=4n$.

Next, we consider the vertex-weights and show that they are pairwise distinct. The vertex-weights of $C_n\times P_2$ is as follows:
$$wt_g(u_i)=\left\{\begin{array}{ll}
W_7\\ \ \\
W_8
\end{array}\right.$$
where $$W_ 7=4(2n+i)+(t-1)$$ with the following conditions:

whenever $i=1$, $t=1$,

whenever $2\leq i\leq \lceil \frac{n}{2}\rceil$, $t=-2$.

and $$W_8=4(3n-i)+(t+5)$$ with the following conditions:

when $i=\frac{n}{2}+1, n$ even, $t=-1$,

when $i=\lceil \frac{n}{2}\rceil +1, n$ odd, $t=1$,

when $\lceil \frac{n}{2} \rceil +2\leq i\leq n$, $t=2$.

Also, $$wt_g(v_i)=\left\{\begin{array}{ll}
W_9\\ \ \\
W_{10}
\end{array}\right.$$
where $$W_9=13n+4i+(t-1)$$ with the following conditions:

$t=1$ whenever $i=1$ and $t=-2$ whenever $2\leq i\leq \lceil \frac{n}{2} \rceil$.\\
Also, $$W_{10}= 17n-4i+(t+5)$$ with the following conditions:

$t=-1$ whenever $i=\frac{n}{2}+1$ and $n$ even,

$t=1$ for $i=\lceil \frac{n}{2}\rceil +1$ and $n$ odd,

$t=2$ whenever $\lceil \frac{n}{2}\rceil +2 \leq i\leq n$.\\
In view of the above, the set $\{W_i\}_{i=7}^{10}$ are pairwise distinct which shows that the labeling $g$ is a vertex-antimagic total.\\

For the edge-weights under the labeling $g$, we have the following:\\
$$wt_g(u_iu_{i+1})=\left\{\begin{array}{ll}
W_{11}\\ \ \\
W_{12}
\end{array}\right.$$
where $$W_{11}=2(n+3i)+(t-5)$$ satisfying the following:

$t=3$ for $i=1$,

$t=2$ for $2\leq i\leq \lfloor\frac{n}{2}\rfloor$,

$t=1$ for $i=\lceil \frac{n}{2} \rceil$, $n$ odd.\\
Also, $$W_{12}=2(4(n+1)-3i)+t$$ satisfying the following conditions:

$t=-3$ whenever $i=\frac{n}{2}+1, n$ even,

$t=-2$ whenever $\lfloor \frac{n}{2}\rfloor +2\leq i\leq n$.
$$wt_g(v_iv_{i+1})=\left\{\begin{array}{ll}
W_{13}\\ \ \\
W_{14}\end{array}\right.$$
Where $$W_{13}=6(n+i)+(t-5)$$ with the following conditions:

for $i=1$, $t=3$,

for $2\leq i\leq \lfloor \frac{n}{2}\rfloor$, $t=3$,

for $i=\lceil \frac{n}{2}\rceil$ and $n$ odd, $t=1$.
Furthermore, $$W_{14}=2(3(2n-1)+4)+t$$ satisfying the followings:

$t=-3$ whenever $i=\frac{n}{2}+1$ and $n$ even,

$t=-2$ whenever $\lfloor \frac{n}{2}\rfloor +2\leq i\leq n$.\\
Moreover, $$wt_g(u_iv_i)=\left\{\begin{array}{ll}
5n+2i+t, & t=0 \ \ for\ \ i=1\ \ and\ \ t=-1 \ \ for\ \  2\leq i\leq \lfloor \frac{n}{2} \rfloor+1\\ \ \\
7n+2(2-i), & \lfloor\frac{n}{2} \rfloor+2\leq i\leq n
\end{array}\right.$$
The edge-weights of the edges in $C_n\times P_2$ under the labeling $g$ are all different which implies that the labeling is edge-antimagic total.
Therefore the labeling $g$ is TAT labeling since it is both vertex-antimagic total and edge-antimagic total. Thus the prism graphs $C_n\times P_2$ is a totally antimagic total graph.

\end{proof}
Next, we give the TAT labeling of generalised Petersen graph.

\begin{thm}
The generalised Petersen graph $P(n,m), n\geq 3$ and $1\leq m\leq \lfloor \frac{n-1}{2}\rfloor$ is a super TAT graph.
\end{thm}
\begin{proof}
Denotes the vertices of the graph $P(n,m)$ by the symbols $u_iv_i, 1\leq i\leq n$. Let $g$ be a labeling on the graph $P(n,m)$ defined in the following way:
$$g(u_i)=i, 1\leq i\leq n$$
$$g(v_i)=\left\{\begin{array}{ll}
n+1, & 1\leq i\leq n-1\\ \ \\
2(3n-i+1), & j= n
\end{array}\right.$$
Also, define labeling $g$ on the edges as follows:
$$g(u_iu_{i+1})=3n-(i-1), 1\leq i\leq n$$
where $i$ is calculated modulo $n$.
$$g(u_iv_i)=\left\{\begin{array}{ll}
3n+1, & i=1\\ \ \\
4n-(i-2), & 2\leq i\leq n
\end{array}\right.$$
$$g(v_iv_{i+m})=5n-(i-1), 1\leq i\leq n$$
with indices $i+m$ taken modulo $n$.\\
It is easy to see that the labeling $g$ is super. Also, the vertex-weights under the labeling $g$ is as follows:
$$wt_g(u_i)=\left\{\begin{array}{ll}
10n+2-i, & i=1\\ \ \\
5(2n+1)-2i, & 2\leq i\leq n
\end{array}\right.$$
$$g(v_iv_{i+1})=\left\{\begin{array}{lll}
14n+2-m, & i=1\\ \ \\
15n-2i+5-t, &t=m+1 \ \ for \ \ 2\leq  i\leq \lfloor\frac{n}{2} \rfloor \ \\ 
 & and \ \ t=-m\ \ for \ \ \lfloor \frac{n}{2}\rfloor+1\leq i\leq n-1\\ \ \\
12n+m+5, & i=n
\end{array}\right.$$
In view of the above labeling, the vertex-weights are pairwise distinct which implies that the labeling $g$ is super vertex-antimagic total.\\
Next we consider the edge-weights of the graph $P(n,m)$ as follows:
$$wt_g(u_iu_{i+1})=\left\{\begin{array}{ll}
3n+i+2, & 1\leq i\leq n-1\\ \ \\
3n+2, & i=n
\end{array}\right.$$
$$wt_g(v_iv_{i+m})=7n+3-t$$ with the following conditions:

$t=-(k+j)$ whenever $1\leq i\leq \lfloor \frac{n}{2}\rfloor$,

$t=k-j+1$ whenever $\lfloor \frac{n}{2}\rfloor +1\leq i\leq n-1$,

$t=k+1$ whenever $i=n$.
Also, $$wt_g(u_iv_i)=\left\{\begin{array}{lll}
4(n+1), & i=1\\ \ \\
5n+3+i, & 2\leq i\leq n-1\\ \ \\
5n+3, & i=n
\end{array}\right.$$
This means that all the edges in $P(n,m)$ have different edge-weights which implies that the graph $P(n,m)$ is an edge-antimagic total graph. Therefore the generalised Petersen graph $P(n,m)$ is a super totally antimagic total graph since it admits both vertex-antimagic total labeling and edge-antimagic total labeling.
\end{proof}
\section{Totally antimagic total labeling of chain graphs}
In this section, we prove that the chain graphs of totally antimagic total graphs is totally antimagic total graph. Suppose now that the graphs $B_1,B_2,...,B_m$ are blocks and that for any $i\in \{1,2,...,m\}$, $B_i$ and $B_{i+1}$ have a vertex in common in such a way that the block-cut point is a path. The graph $G$ obtained by the concatenation will be called a \emph{chain} graph.
\begin{thm}\label{RT}
The chain graph $G$ obtained by concatenation of totally antimagic total graphs is a TAT graph.
\end{thm}
\begin{proof}
Let $G$ denotes that chain graph with blocks $G_i, 1\leq i\leq m$. Let $g_i,1\leq i\leq m$ be a TAT labeling of $G_i$. Also, let $p=|V(G_i)|$ and $$g_i:V(G_i)\cup E(G_i)\rightarrow \{1,2,...,|V(G_i)|+|E(G_i)|\}$$ such that $wt_{g_i}(v)\neq wt_{g_i}(u)$ for all $u,v\in V(G_i),u\neq v$ and $wt_{g_i}(e)\neq wt_{g_i}(h)$ for all $e,h\in E(G_i), e\neq h$.\\
Consider the cut-vertex between $G_i$ and $G_{i+1}$ as the concatenation of the vertex labeled with $1\in V(G_i)$ and vertex labeled with $p\in V(G_{i+1})$ denoted by $r_i$. Without loss of generality we may assume that $r_{i+1}>r_i$. Define a labeling for $G$ such that
$$f(x)=\left\{\begin{array}{ll}
g_1(x), & x\in V(G_1)\\ \ \\
g_i(x)+\disp{\sum_{j=1}^{i-1}|V(G_j)|+\sum_{j=1}^{i-1}|E(G_j)|-(i-2)}, & x\in V(G_i),i=1,2,...,m
\end{array}\right.$$
and $$f(e)=\left\{\begin{array}{ll}
g_1(e), & e\in E(G_1)\\ \ \\
g_i(e)+\disp {\sum_{j=1}^{i=1}|V(G_j)|+\sum_{j=1}^{i-1}|E(G_j)|-(i-1)}, & e\in E(G_i),i=1,2,...,m
\end{array}\right.$$
It is easy to see that $f$ is a total antimagic total labeling of $G$. For the edge-weights under the labeling $f$, we obtain
$$wt_f(e)=\left\{\begin{array}{ll}
wt_{g_1}(e), & e\in E(G_1)\\ \ \\
wt_{g_i}(e)+3(\disp {\sum_{j=1}^{i-1}|V(G_j)|+\sum_{j=1}^{i-1}|E(G_j)|})-1, & e\in E(G_i),i=1,2,...,m
\end{array}\right.$$
Also, the edge-weights for the edges in $G_{i+1}$ incidents with the cut-vertex $r_i$, under the labeling $f$ is as follows:
$$wt_f(e)=\left\{\begin{array}{ll}
wt_{g_1}(e), & e\in E(G_1)\\ \ \\
wt{g_i}(e)+2(\disp {\sum_{j=1}^{i-1}|V(G_j)|+\sum_{j=1}^{i-1}|E(G_j)|})-p+r_i-r_1, & e\in E(G_i),i=1,2,...,m
\end{array}\right.$$
As $g_i,i=1,2,...,m$ is edge-antimagic labeling, the edge-weights of all edges in $G$ under the labeling $f$ are pairwise distinct.\\
The maximum edge-weight of an edge $e\in E(G_i),1\leq i\leq m$ is
$$wt_f^{max}(e)\leq 3(\disp {\sum_{j=1}^{i}|V(G_j)|+\sum_{j=1}^{i-1}|E(G_j)|})-|V(G_1)|+|E(G_1)|, e\in E(G_i),i=1,2,...,m.$$
Thus $f$ is an edge-antimagic labeling of $G$.\\
For the vertex-weight under labeling $f$, we get
$$wt_f(v)=\left\{\begin{array}{ll}
wt_{g_1}(v), & v\in V(G_1)\\ \ \\
wt{g_i}(v)+(deg(v)_+1)(\disp {\sum_{j=1}^{i-1}|V(G_j)|+\sum_{j=1}^{i-1}|E(G_j)|}-1)+1, & e\in E(G_i),i=1,2,...,m
\end{array}\right.$$
For the cut-vertices $r_i$, the weights under labeling $f$, for cycle-like structure, we get
$$wt_f(r_i)=wt_{g_1}(v_p)+wt_{g_1}(v_1)+deg(v)(\disp {\sum_{j=1}^{i-1}|V(G_j)|+|E(G_j)|})+$$
$$(deg(v)+1)(\disp{\sum_{j=1}^{i-2}|V(G_j)+E(G_j)|})-deg(r_i)-p.$$
where $v_p$ and $v_1$ are the vertices in $G_1$ labeled with $p$ and $1$ respectively.\\
Also, the cut-vertices weights under labeling $f$ for path-like structure is as follows:
$$wt_f(r_i)=wt_{g_1}(v_p)+wt_{g_1}(v_1)+(deg(v)+1)(2(\sum_{j=1}^{i-1}|V(G_j)|+|E(G_j)|)$$$$+|V(G_{j-1})|+|E(G_{j-1})|)-deg(r_i)-p$$
In view of the above labeling, the chain graph $G$ is a TAT graph.
\end{proof}
\begin{cor}
The tree graph formed from the concatenation of paths is a TAT graph.
\end{cor}
\begin{proof}
This follows from directly from theorem \ref{RT}.
\end{proof}
\emph{Conjecture}: All trees are TAT graph.


\begin{thebibliography}{99}
\bibitem{AJ}A. D. Akwu, D. O. A. Ajayi, On totally antimagic total labeling of bipartite graphs, arXiv:1601.02112v2 [math.CO].
\bibitem{AK}N. Alon, G. Kaplan, A. Lev, Y. Roditty and R. Yuster, Dense graphs are antimagic, J. Graph Theory 47, (2004), 297-309.
\bibitem{BE1}M. Baca,.M. Miller, O. Phanalsy, J. Ryan, A. Semanicova-Fenovcikora and A. A. Sillasen, Totally antimagic total graphs, Aust. J. Combin., 61 (2015)42-56.
\bibitem{BAM}M. Baca and M. Miller, Super Edge-antimagic Graphs, Brownwalker Press, Boca Raton, 2008.
    \bibitem{GA}J. GAllian, A dynamic survey of graph labeling, Electronic J. Combin. 16 (2013), DS6.
    \bibitem{HR}N. Hartsfield and G. Ringel, Pearls in Graph Theory, Academic Press, Boston-san Diego-New York, London, 1990.
    \bibitem{MD}A. M. Marr and W. D. Wallis, Magic-Graphs, Second Ed., Birkhauser, New York, 2013.
\bibitem{MPJ}O. Phanalsy and J. Ryan, All graphs have antimagic total labeling, Electron Notes in Discrete Math. (2011), 645-650.
    \bibitem{WA}W. D. Wallis, Magic graphs, Birkhauser, Boston, Basel, Berlin 2001.
\bibitem{WE}D.B. West, An introduction to Graph Theory, Prentice-Hall Englewood cliffs, NJ. 1996.

 \end{thebibliography}
\end{document}